\documentclass[12pt, reqno, twoside, letterpaper]{amsart}

\usepackage{zeta_vinogradov}

\title[Zeta functions and asymptotic additive bases]
      {Zeta functions and asymptotic additive bases\\
      with some unusual sets of primes}
      
\author[W.\ D.\ Banks]{William D.\ Banks}

\address{Department of Mathematics, 
         University of Missouri, 
         Columbia MO, USA.}

\email{bankswd@missouri.edu}
        
\date{\today}

\begin{document}

\begin{abstract}
Fix $\delta\in(0,1]$, $\sigma_0\in[0,1)$ and a real-valued function $\eps(x)$
for which $\varlimsup_{x\to\infty}\eps(x)\le 0$.
For every set of primes $\cP$ whose counting function $\pi_\cP(x)$
satisfies an estimate of the form
$$\pi_\cP(x)=\delta\,\pi(x)+O\bigl(x^{\sigma_0+\eps(x)}\bigr),$$
we define a zeta function $\zeta_\cP(s)$ that is
closely related to the Riemann zeta function $\zeta(s)$.  For
$\sigma_0\le\frac12$, we show that
the Riemann hypothesis is equivalent to the non-vanishing of
$\zeta_\cP(s)$ in the region $\{\sigma>\frac12\}$.

For every set of primes $\cP$ that contains the prime $2$ and
whose counting function satisfies an estimate of the form
$$\pi_\cP(x)=\delta\,\pi(x)+O\bigl((\log\log x)^{\eps(x)}\bigr),$$ 
we show that $\cP$ is an \emph{exact}
asymptotic additive basis for $\NN$,
i.e., for some integer $h=h(\cP)>0$ the
sumset $h\cP$ contains all but finitely many natural numbers.
For example, an exact asymptotic additive basis for $\NN$ is provided
by the set
$$
\{2,547,1229,1993,2749,3581,4421,5281\ldots\},
$$
which consists of $2$ and every hundredth prime thereafter.
\end{abstract}

\maketitle


\section{Introduction and statement of results}

Let $\NN$ denote the set of positive integers and $\PP$ the set of prime numbers.
Denote by $\pi(x)$ the prime counting function
$$
\pi(x)\defeq\#\{p\le x:p\in\PP\},
$$
and for any given set of primes $\cP$, put
$$
\pi_\cP(x)\defeq\#\{p\le x:p\in\cP\}.
$$
Given $\delta\in(0,1]$, $\sigma_0\in[0,1)$
and a real function $\eps(x)$
such that $\varlimsup\limits_{x\to\infty}\eps(x)\le 0$, let
$\sA(\delta,\sigma_0,\eps)$ denote the class consisting of sets
$\cP\subseteq\PP$ for which one has an estimate of the form
\begin{equation}
\label{eq:greenearth_zeta}
\pi_\cP(x)=\delta\,\pi(x)+O\bigl(x^{\sigma_0+\eps(x)}\bigr),
\end{equation}
where the implied constant may depend on $\cP$. Let
$\sB(\delta,\eps)$ denote the class consisting of sets $\cP\subseteq\PP$ such that
\begin{equation}
\label{eq:greenearth_asymp}
\pi_\cP(x)=\delta\,\pi(x)+O\bigl((\log\log x)^{\eps(x)}\bigr),
\end{equation}
where again the implied constant may depend on $\cP$.
The aim of this paper is to state some general results that hold true
for all sets in $\sA(\delta,\sigma_0,\eps)$,
or for all sets in $\sB(\delta,\eps)$.
We also give examples of sets $\cP$ in these classes,
to which our general results can be applied.

\subsection{Analogues of the Riemann zeta function}
\label{sec:analogues}

The Riemann zeta function is defined
in the half-plane $\{s=\sigma+it\in\CC:\sigma>1\}$
by two equivalent expressions, namely
$$
\zeta(s)\defeq\sum_{n\in\NN} n^{-s}
=\prod_{p\in\PP}(1-p^{-s})^{-1}\qquad(\sigma>1).
$$
In the extraordinary memoir of Riemann~\cite{Riemann} it is shown
that $\zeta(s)$ extends to a meromorphic function on the complex plane,
its only singularity being a simple pole at $s=1$, and that it satisfies a
functional equation relating its values at $s$ and $1-s$.
The Riemann hypothesis (RH) asserts that every non-real zero of $\zeta(s)$
lies on the critical line $\{\sigma=\tfrac12\}$.

Although the function $\zeta(s)$ incorporates all of the primes
into its definition, in this paper we observe that certain thin subsets
of the primes also give rise to functions that are strikingly similar to
$\zeta(s)$.

\begin{theorem}
\label{thm:main_zeta}
For any set $\cP\in\sA(\delta,\sigma_0,\eps)$, the function $\zeta_\cP(s)$ defined by
$$
\zeta_\cP(s)\defeq\prod_{p\in\cP}(1-p^{-s})^{-1/\delta}
\qquad(\sigma>1)
$$
extends to a meromorphic function on the region
$\{\sigma>\sigma_0\}$, and there is a function $f_\cP(s)$ which is analytic on
$\{\sigma>\sigma_0\}$ and has the property that
\begin{equation}
\label{eq:zeta_relation}
\zeta_\cP(s)=\zeta(s)\exp(f_\cP(s))
\qquad(\sigma>\sigma_0).
\end{equation}
\end{theorem}

This is proved in \S\ref{sec:proof_main_zeta} below.

The following corollary is clear in view of \eqref{eq:zeta_relation};
it shows that the truth of the Riemann hypothesis depends
only on the distributional properties of certain (potentially thin)
sets of primes.

\begin{corollary}
\label{cor:main_zeta}
If $\cP\in\sA(\delta,\sigma_0,\eps)$ and $\sigma_0\le\tfrac12$,
then the Riemann hypothesis is true if and only if $\zeta_\cP(s)\ne 0$
in the half-plane $\{\sigma>\tfrac12\}$.
\end{corollary}

Similarly, for every nontrivial primitive Dirichlet character $\chi$,
the Dirichlet $L$-function $L(s,\chi)$, which is initially defined by
$$
L(s,\chi)\defeq\sum_{n\in\NN} \chi(n)n^{-s}
=\prod_{p\in\PP}(1-\chi(p)p^{-s})^{-1}\qquad(\sigma>1),
$$
extends to an entire function on the complex plane and satisfies a
functional equation relating its values at $s$ and $1-s$. The generalized
Riemann hypothesis (GRH) asserts that every non-real zero of $L(s,\chi)$
lies on the critical line.

The following result provides (in some cases) an analogue of Theorem~\ref{thm:main_zeta}.
It is proved only for quadratic Dirichlet characters $\chi$.
For any such character, let us denote
$$
\pi^-(x;\chi)\defeq\#\big\{p\le x:p\in\PP\text{~and~}\chi(p)=-1\big\},
$$
and for a given set of primes $\cP$, put
$$
\pi_\cP^-(x;\chi)\defeq\#\big\{p\le x:p\in\cP\text{~and~}\chi(p)=-1\big\}.
$$

\begin{theorem}
\label{thm:main_Lfun}
Fix $\cP\in\sA(\delta,\sigma_0,\eps)$.
Let $\chi$ be a primitive quadratic Dirichlet character,
and suppose that
\begin{equation}
\label{eq:greenearth_quadratic}
\pi_\cP^-(x;\chi)=\rho\,\pi^-(x;\chi)+O(x^{\sigma_0+\eps(x)}),
\end{equation}
where $\rho\in(0,1]$. Suppose further that $\rho/\delta=A/B$ for two positive
integers $A,B$. Then, the function $L_\cP(s,\chi)$ defined by
$$
L_\cP(s,\chi)\defeq
\prod_{p\in\cP}(1-\chi(p)p^{-s})^{-1/\delta}
\qquad(\sigma>1)
$$
extends to a meromorphic function on the region
$\{\sigma>\sigma_0\}$, and 
there is a function $f_\cP(s,\chi)$ which is analytic on
$\{\sigma>\sigma_0\}$ and has the property that
$$
\zeta(s)^B L_\cP(s,\chi)^B=\zeta(s)^A L(s,\chi)^A\exp(f_\cP(s,\chi))
\qquad(\sigma>\sigma_0).
$$
\end{theorem}

This is proved in \S\ref{sec:proof_main_Lfun} below.

\subsection{Remarks}
\label{sec:remarks_zeta}

If one assumes that $\eps(x)$ is such that the integral
$\int_1^\infty x^{\eps(x)-1}\,dx$ converges
(for example, $\eps(x)\defeq -2(\log\log 2x)/\log x$), then
$\zeta_\cP(s)$ and $f_\cP(s)$ in Theorem~\ref{thm:main_zeta}
extend to continuous functions in the closed half-plane
$\{\sigma\ge\sigma_0\}$, and the relation
\eqref{eq:zeta_relation} persists throughout
$\{\sigma\ge\sigma_0\}$. For such $\eps(x)$ one can easily deduce the following omega result in the
case that $\sigma_0=\tfrac12$.

\begin{corollary}
\label{cor:omega_result}
Let $\kappa:\PP\to\{\pm 1\}$ be a function that satisfies the
estimate
$$
\#\big\{\text{\rm prime~}p\le x:\kappa(p)=-1\big\}=\tfrac12\pi(x)+O(x^{1/2+\eps(x)}).
$$
Then, for any primitive quadratic Dirichlet character $\chi$ we have
$$
\#\big\{\text{\rm prime~}p\le x:\chi(p)=\kappa(p)\big\}
=\Omega(x^{1/2+\eps(x)}).
$$
\end{corollary}

\noindent However, a stronger (and considerably
more general) result has been obtained by
Kisilevsky and Rubinstein~\cite{KisRub}.  Their work 
lies much deeper and utilizes explicit information
about zeros of $L$-functions.

\subsection{Examples}
\label{sec:examples_zeta}

Here, we illustrate the results stated in \S\ref{sec:analogues}
with some special sets of primes.

Let $p_n$ denote the $n$th smallest prime number for
each positive integer $n$.  Note that $n=\pi(p_n)$, thus $\pi(p)$ is
the index associated to any given prime $p$.
Let $\PP_{k,b}$ denote the set of primes whose index lies in a fixed arithmetic progression $b\bmod k$; that is,
$$
\PP_{k,b}\defeq\big\{p\in\PP:\pi(p)\equiv b\bmod k\big\}.
$$

Let $\cP\defeq\PP_{k,b}$.
Since $p_n\le x$ if and only if $n\le\pi(x)$, we have
\begin{equation}
\label{eq:piPest}
\pi_\cP(x)=\#\{n\le\pi(x):n\equiv b\bmod k\}
=\fl{\frac{\pi(x)-b}{k}}=\frac{1}{k}\,\pi(x)+O(1),
\end{equation}
where $\fl{\cdot}$ is the floor function;
this shows that \eqref{eq:greenearth_zeta}
holds with $\delta=\frac1k$, $\sigma_0=0$, and $\eps(x)\equiv 0$;
in other words, $\PP_{k,b}\in\sA(\tfrac1k,0,0)$.
Applying Theorem~\ref{thm:main_zeta} and Corollary~\ref{cor:main_zeta}
we immediately deduce the following.

\begin{corollary}
\label{cor:main_zetakb}
The function
$$
\zeta_{k,b}(s)\defeq\prod_{p\in\PP_{k,b}}
(1-p^{-s})^{-k}\qquad(\sigma>1).
$$
extends to a meromorphic function on the region $\{\sigma>0\}$,
and there is a function $f_{k,b}(s)$ which is analytic on $\{\sigma>0\}$
and has the property that
$$
\zeta_{k,b}(s)=\zeta(s)\exp(f_{k,b}(s))\qquad(\sigma>0).
$$
Consequently, the Riemann hypothesis is true
if and only if $\zeta_{k,b}(s)\ne 0$ in $\{\sigma>\tfrac12\}$.
\end{corollary}

This shows that much analytic information about the Riemann zeta function
(in particular, the location of the nontrivial zeros) is captured
by a set of primes of relative density $\frac1k$.

More generally, for fixed $\kappa,\lambda\in\RR$ with
$\kappa\ge 1$, let $\cB_{\kappa,\lambda}$ be the 
non-homogeneous Beatty sequence defined by
$$
\cB_{\kappa,\lambda}\defeq\big\{n\in\NN:n=\fl{\kappa m+\lambda}\text{~for some~}m\in\ZZ\big\}.
$$
Beatty sequences appear in a variety of mathematical
settings; the arithmetic properties of these sequences have been
extensively explored in the literature.
Let $\PP_{\kappa,\lambda}$ denote the set of primes whose index lies in
$\cB_{\kappa,\lambda}$; that is,
$$
\PP_{\kappa,\lambda}\defeq\big\{p\in\PP:\pi(p)\in\cB_{\kappa,\lambda}\big\}.
$$
As with \eqref{eq:piPest} above, the estimate
$$
\pi_\cP(x)=\tfrac{1}{\kappa}\pi(x)+O(1)
$$
is immediate; therefore, $\PP_{\kappa,\lambda}\in\sA(\tfrac1\kappa,0,0)$,
and one obtains a natural extension of Corollary~\ref{cor:main_zetakb} with the function
$$
\zeta_{\kappa,\lambda}(s)\defeq\prod_{p\in\PP_{\kappa,\lambda}}
(1-p^{-s})^{-\kappa}.
$$

Next, let $\bfX\defeq\{\bfX_p:p\in\PP\}$ be a set of independent random variables,
where each variable is either $+1$ or $-1$, with a 50\% probability for either value.
The law of the iterated logarithm (due to Khintchine~\cite{Khintchine})
asserts that
$$
\varlimsup_{x\to\infty}\big(\pi(x)\log\log\pi(x)\big)^{-1/2}
\sum_{p\le x}\bfX_p=\sqrt{2}\qquad\text{a.s.}
$$
and (replacing $\{\bfX_p\}$ with $\{-\bfX_p\}$) that
$$
\varliminf_{x\to\infty}\big(\pi(x)\log\log\pi(x)\big)^{-1/2}
\sum_{p\le x}\bfX_p=-\sqrt{2}\qquad\text{a.s.},
$$
where ``a.s.'' stands for ``almost surely'' in the sense of probability theory.
In particular, denoting
$$
\PP_\bfX^+\defeq\big\{p\in\PP:\bfX_p=+1\big\}\mand
\PP_\bfX^-\defeq\big\{p\in\PP:\bfX_p=-1\big\},
$$
we have the (less precise) estimate
$$
\pi_{\PP_\bfX^+}(x)-\pi_{\PP_\bfX^-}(x)
=\sum_{p\le x}\bfX_p=O(x^{1/2})\qquad\text{a.s.}
$$
Since $\pi(x)=\pi_{\PP_\bfX^+}(x)+\pi_{\PP_\bfX^-}(x)$ we deduce that
$$
\pi_{\PP_\bfX^\pm}(x)=\tfrac12\,\pi(x)+O(x^{1/2})\qquad\text{a.s.}
$$
for either choice of the sign $\pm$.  Taking $\cP\defeq\PP_\bfX^\pm$
we see that \eqref{eq:greenearth_zeta} holds a.s.\ with
$\delta=\sigma_0=\frac12$ and $\eps(x)\equiv 0$; in other words,
$\PP_\bfX^\pm\in\sA(\tfrac12,\tfrac12,0)$ almost surely.
In view of Theorem~\ref{thm:main_zeta} and Corollary~\ref{cor:main_zeta}
we deduce the following.
 
\begin{corollary}
\label{cor:main_zetaX}
In the region $\{\sigma>1\}$, let 
\begin{equation}
\label{eq:zetaX}
\zeta_\bfX^+(s)\defeq\prod_{\substack{p\in\PP\\\bfX_p=+1}}(1-p^{-s})^{-2}\mand
\zeta_\bfX^-(s)\defeq\prod_{\substack{p\in\PP\\\bfX_p=-1}}(1-p^{-s})^{-2}
\end{equation}
Then, almost surely, both functions $\zeta_\bfX^\pm(s)$ 
extend to meromorphic functions on the region $\{\sigma>\tfrac12\}$,
and there are functions $f_\bfX^\pm(s)$
which are analytic on $\{\sigma>\tfrac12\}$ and are such that
\begin{equation}
\label{eq:zetaX_relation}
\zeta_\bfX^\pm(s)=\zeta(s)\exp(f_\bfX^\pm(s))\qquad(\sigma>\tfrac12).
\end{equation}
The Riemann hypothesis is equivalent to the assertion that,
almost surely,
$\zeta_\bfX^\pm(s)\ne 0$ in $\{\sigma>\tfrac12\}$
for either choice of the sign $\pm$.
\end{corollary}

In the region $\{\sigma>1\}$, let us now define
\begin{equation}
\label{eq:LfunX}
L(s,\bfX)\defeq\prod_{p\in\PP}(1-\bfX_p\,p^{-s})^{-1}.
\end{equation}
The next corollary reproduces a result
that was first proved by Wintner~\cite{Wintner}
and laid the foundation for random multiplicative
function theory; it asserts that the GRH almost surely holds for the
``$L$-function'' $L(s,\bfX)$  (for more modern work in this direction, see 
\cite{ChatSound,GranSound,Harper,HarNikRad,LauTenWu}).

\begin{corollary}
\label{cor:main_zetaX2}
The function $L(s,\bfX)$
almost surely extends to an analytic function without zeros
in the region $\{\sigma>\tfrac12\}$.
\end{corollary}

Indeed, using \eqref{eq:zetaX} and \eqref{eq:LfunX} we have
\begin{align*}
L(s,\bfX)^2
&=\prod_{p\in\PP}(1-\bfX_p\,p^{-s})^{-2}\\
&=\prod_{p\in\PP_\bfX^+}(1+p^{-s})^{-2}
\prod_{p\in\PP_\bfX^-}(1-p^{-s})^{-2}\\
&=\prod_{p\in\PP_\bfX^+}(1-p^{-2s})^{-2}
\prod_{p\in\PP_\bfX^+}(1-p^{-s})^2
\prod_{p\in\PP_\bfX^-}(1-p^{-s})^{-2}\\
&=\zeta_\bfX^+(2s)\zeta_\bfX^+(s)^{-1}\zeta_\bfX^-(s),
\end{align*}
By Corollary~\ref{cor:main_zetaX} there are (almost surely)
functions $f_\bfX^\pm(s)$ which are analytic on $\{\sigma>\tfrac12\}$
and satisfy
\eqref{eq:zetaX_relation}; in particular, the relation
\begin{align*}
L(s,\bfX)^2=\zeta_\bfX^+(2s)\zeta_\bfX^+(s)^{-1}\zeta_\bfX^-(s)
=\zeta(2s)\exp\big(f_\bfX^+(2s)-f_\bfX^+(s)+f_\bfX^-(s)\big)
\end{align*}
holds in $\{\sigma>1\}$, and it provides the required analytic continuation
of $L(s,\bfX)$ to the region $\{\sigma>\frac12\}$. Moreover,
$L(s,\bfX)\ne 0$ in $\{\sigma>\frac12\}$.

\subsection{Asymptotic additive bases}
\label{sec:asymp}

\begin{theorem}
\label{thm:main_asymp}
Every set $\cP\in\sB(\delta,\eps)$ containing the prime $2$
is an exact asymptotic additive basis for $\NN$.
In other words, there is an integer $h=h(\cP)>0$ such that the $h$-fold sumset
$$
h\cP\defeq
\mathop{\underbracket{\hskip3pt\cP+\cdots+\cP\hskip1pt}}\limits_{\text{$h$ copies}}
$$
contains all but finitely many natural numbers.
\end{theorem}

This is proved in \S\ref{sec:proof_main_asymp} below.
We remark that S\'ark\"ozy~\cite{Sar} has shown that
any set of primes $\cP$ is an asymptotic additive basis
for $\NN$, and stronger quantitative versions have been
obtained; see \cite{LiPan,Mat,RamRus,Shao}.  To show
that every $\cP\in\sB(\delta,\eps)$ containing $2$
is an \emph{exact} asymptotic additive basis,
we use a result of Shiu~\cite{Shiu} on strings of consecutive primes in an
arithmetic progression; in principle, the methods of
Green and Tao~\cite{GreTao} could be used to prove
Theorem~\ref{thm:main_asymp} with $\sB(\delta,\eps)$
replaced with a rather more restricted class of
prime sets.

\subsection{Examples}
\label{sec:examples_asymp}

As in \S\ref{sec:examples_zeta}, we put
$$
\PP_{k,b}\defeq\big\{p\in\PP:\pi(p)\equiv b\bmod k\big\}.
$$
We have already seen that
$$
\pi_\cP(x)=\tfrac{1}{k}\pi(x)+O(1)
$$
holds with $\cP\defeq\PP_{k,b}$, and therefore $\PP_{k,b}\in\sB(\tfrac1k,0)$.
Since $2\in\PP_{k,b}$ if and only if $b=1$, the next corollary follows
immediately from Theorem~\ref{thm:main_asymp}.

\begin{corollary}
\label{cor:main_asymp1}
For every $k\in\NN$, the set $\PP_{k,1}$ is an
exact asymptotic additive basis for $\NN$.
For all $b,k\in\NN$, the set $\PP_{k,b}\cup\{2\}$ is an
exact asymptotic additive basis for $\NN$.
\end{corollary}

For example, an exact asymptotic additive basis for $\NN$ is provided
by the set
$$
\PP_{100,1}=\{2,547,1229,1993,2749,3581,4421,5281\ldots\},
$$
which consists of $2$ and every hundredth prime thereafter.

More generally, for the set $\PP_{\kappa,\lambda}$ defined in \S\ref{sec:examples_zeta},
we have the following result.

\begin{corollary}
\label{cor:main_asymp1}
For any $\kappa,\lambda\in\RR$ with $\kappa\ge 1$,
the set $\PP_{\kappa,\lambda}\cup\{2\}$ is an
exact asymptotic additive basis for $\NN$.
\end{corollary}

\section{Proof of Theorem~\ref{thm:main_zeta}}
\label{sec:proof_main_zeta}

Suppose first that $s\in\CC$ with $\sigma>1$.
From the Euler product representations of
$\zeta_\cP(s)$ and $\zeta(s)$ we see that the function
$$
f_\cP(s)\defeq\log\zeta_\cP(s)-\log\zeta(s)
$$
can be written in the form
$$
f_\cP(s)=\sum_{j\ge 1}j^{-1}f_{\cP,j}(s)
$$
with
\begin{equation}
\label{eq:fPjs_defn}
f_{\cP,j}(s)\defeq \frac1\delta\sum_{p\in\cP}p^{-js}-\sum_{p\in\PP}p^{-js}
\qquad(j\ge 1).
\end{equation}
To prove the theorem, it is enough to show that $f_\cP(s)$
extends to an analytic function in $\{\sigma>\sigma_1\}$ for every
 real number $\sigma_1>\sigma_0$.

Let $\sigma_1$ be given.  Noting that $\sigma_1>0$, let
$N$ be a positive integer such that $\sigma_1>\frac1N$.
It is easy to verify that
$$
\sum_{j>N}j^{-1}f_{\cP,j}(s)
$$
extends to an analytic function in $\{\sigma>\frac1N\}$, hence also
in $\{\sigma>\sigma_1\}$. Therefore, it remains to show that
for any fixed $j\in[1,N]$, $f_{\cP,j}(s)$ extends to an
analytic function in $\{\sigma>\sigma_1\}$.

Using \eqref{eq:greenearth_zeta} we have
$$
\frac1\delta\,\pi_\cP(u)=\pi(u)+E(u)\qquad(u\ge 1),
$$
where $E(u)\ll u^{\sigma_0+\eps(u)}$, and therefore
\begin{align*}
\frac1\delta\sum_{p\in\cP}p^{-js}&=\frac1\delta\int_{1}^\infty u^{-js}\,d\pi_\cP(u)
=\frac{js}\delta\int_{1}^\infty u^{-js-1}\pi_\cP(u)\,du\\
&=js\int_{1}^\infty u^{-js-1}\pi(u)\,du
+js\int_{1}^\infty u^{-js-1}E(u)\,du\\
&=\sum_{p\in\PP}p^{-js}
+js\int_{1}^\infty u^{-js-1}E(u)\,du;
\end{align*}
that is,
$$
f_{\cP,j}(s)=js\int_{1}^\infty u^{-js-1}E(u)\,du.
$$
Since $E(u)\ll u^{\sigma_0+\eps(u)}$, the latter integral converges absolutely
in $\{\sigma>j^{-1}\sigma_0\}$, hence also in $\{\sigma>\sigma_1\}$,
and the integral representation provides the required analytic extension
of $f_{\cP,j}(s)$ when $j\in[1,N]$.

\section{Proof of Theorem~\ref{thm:main_Lfun}}
\label{sec:proof_main_Lfun}

As in \S\ref{sec:proof_main_zeta} we first 
assume that $s\in\CC$ with $\sigma>1$ and define
$$
f_\cP(s,\chi)\defeq B\log(\zeta(s)L_\cP(s,\chi))
-A\log(\zeta(s)L(s,\chi))=\sum_{j\ge 1}j^{-1}f_{\cP,j}(s,\chi),
$$
where
$$
f_{\cP,j}(s,\chi)\defeq \frac B\delta\sum_{p\in\cP}\chi(p)^jp^{-js}
-A\sum_{p\in\PP}\chi(p)^jp^{-js}+
(B-A)\sum_{p\in\PP}p^{-js}\qquad(j\ge 1).
$$
As before, let $\sigma_1>\sigma_0$ be given, and let
$N$ be a fixed positive integer such that $\sigma_1>\frac1N$.
To prove the theorem, it is enough to show that for any fixed $j\in[1,N]$
the function $f_{\cP,j}(s,\chi)$ has an analytic extension to the
region $\{\sigma>\sigma_1\}$.

Put
\begin{align*}
f_1(s)&\defeq\frac B\delta\sum_{\substack{p\in\cP\\\chi(p)=1}}
p^{-js}-A\sum_{\substack{p\in\PP\\\chi(p)=1}}p^{-js},\\
f_2(s)&\defeq\frac B\delta\sum_{\substack{p\in\cP\\\chi(p)=-1}}
p^{-js}-A\sum_{\substack{p\in\PP\\\chi(p)=-1}}p^{-js},\\
f_3(s)&\defeq\frac B\delta\sum_{\substack{p\in\cP\\p\mid q}}
p^{-js}-A\sum_{\substack{p\in\PP\\p\mid q}}p^{-js},\\
f_4(s)&\defeq (B-A)\sum_{p\in\PP}p^{-js},
\end{align*}
where $q$ is the modulus of the character $\chi$.  We have
$$
f_1(s)+f_2(s)+f_3(s)+f_4(s)=
\frac B\delta\sum_{p\in\cP}p^{-js}-B\sum_{p\in\PP}p^{-js}
=B\,f_{\cP,j}(s),
$$
where $f_{\cP,j}(s)$ is given by \eqref{eq:fPjs_defn}. Recall that
in \S\ref{sec:proof_main_zeta} we have shown that $f_{\cP,j}(s)$
has an analytic extension to the region $\{\sigma>\sigma_1\}$;
the same is also true of $f_3(s)$ (which is clearly entire).
Now observe that
$$
f_{\cP,j}(s,\chi)=f_1(s)+(-1)^j f_2(s)+f_4(s),
$$
and therefore
$$
f_{\cP,j}(s,\chi)=\begin{cases}
-f_3(s)+B\,f_{\cP,j}(s)&\quad\hbox{if $j$ is even},\\
2f_2(s)-f_3(s)+B\,f_{\cP,j}(s)&\quad\hbox{if $j$ is odd}.
\end{cases}
$$
To conclude the proof it remains to show that $f_2(s)$ extends
analytically to the region $\{\sigma>\sigma_1\}$.

Since $\rho/\delta=A/B$ we have
$$
f_2(s)=\frac A\rho\sum_{\substack{p\in\cP\\\chi(p)=-1}}
p^{-js}-A\sum_{\substack{p\in\PP\\\chi(p)=-1}}p^{-js}.
$$
Using \eqref{eq:greenearth_quadratic} we can write
$$
\frac A\rho\,\pi_\cP(u,\chi)=A\,\pi(u,\chi)+E(u),
$$
where $E(u)\ll u^{\sigma_0+\eps(u)}$. Then
\begin{align*}
\frac A\rho\sum_{\substack{p\in\cP\\\chi(p)=-1}}p^{-js}
&=\frac A\rho\int_{1}^\infty u^{-js}\,d\pi_\cP^-(u,\chi)
=\frac{jsA}\rho\int_{1}^\infty u^{-js-1}\pi_\cP^-(u,\chi)\,du\\
&=jsA\int_{1}^\infty u^{-js-1}\pi^-(u,\chi)\,du
+js\int_{1}^\infty u^{-js-1}E(u)\,du\\
&=A\sum_{\substack{p\in\PP\\\chi(p)=1}}p^{-js}
+js\int_{1}^\infty u^{-js-1}E(u)\,du;
\end{align*}
in other words,
$$
f_2(s)=js\int_{1}^\infty u^{-js-1}E(u)\,du.
$$
Since $E(u)\ll u^{\sigma_0+\eps(u)}$, the integral representation
yields the desired analytic continuation
of $f_2(s)$ to $\{\sigma>\sigma_1\}$.

\section{Proof of Theorem~\ref{thm:main_asymp}}
\label{sec:proof_main_asymp}

For the proof of Theorem~\ref{thm:main_asymp}, we use
the following result of Banks, G\"ulo\u glu and
Vaughan~\cite[Theorem~1.2]{BanGulVau} (the proof of which
relies a deep theorem of Kneser;
see Halberstam and Roth~\cite[Chapter~I, Theorem~18]{HalbRoth}).

\begin{lemma}
\label{lem:BanGulVau}
Let $\cP$ be a set of prime numbers such that
$$
\varliminf_{x\to\infty}\frac{\pi_\cP(x)}{x/\log x}>0.
$$
Suppose that there is a number $s_1$ such that for all $s\ge s_1$
and $a,b\in\NN$, the congruence
$$
p_1+\cdots+p_s\equiv a\bmod b
$$
has a solution with $p_1,\ldots,p_s\in\cP$.  Then, there is an integer $h=h(\cP)>0$ such that the $h$-fold sumset $h\cP$
contains all but finitely many natural numbers.
\end{lemma}

We also use the following statement concerning
consecutive primes in a given arithmetic progression,
which is due to Shiu~\cite[Theorem~1]{Shiu}; see also
Banks~\emph{et al}~\cite[Corollary~4]{BanFreTurB}, where
a bounded gaps variant is obtained as
a consequence of the Maynard-Tao theorem (see \cite{Maynard}).

\begin{lemma}
\label{lem:shiustrings}
Let $p_n$ denote the $n$th smallest prime number for
each positive integer $n$. Fix $c,d\in\NN$ with $\gcd(c,d)=1$.  Then, there are infinitely many $r\in\NN$
such that $p_{r+1}\equiv p_{r+2}\equiv\cdots\equiv p_{r+m(r)}\equiv c\bmod d$, where $m(r)$ is an integer-valued
function satisfying the lower bound
\begin{equation}
\label{eq:m(r)bd}
m(r)\gg\left(\frac{\log\log r\,\log\log\log\log r}
{(\log\log\log r)^2}\right)^{1/\phi(d)}.
\end{equation}
Here, $\phi(\cdot)$ is the Euler function.
\end{lemma}

We now make an important observation based on Lemma~\ref{lem:shiustrings},
which may be of independent interest.

\begin{proposition}
\label{prop:statue}
Fix $\cP\in\sB(\delta,\eps)$.
For all $c,d\in\NN$ with $\gcd(c,d)=1$, the set $\cP$
contains infinitely many primes in the arithmetic
progression $c\bmod d$.
\end{proposition}

\begin{proof}
According to Lemma~\ref{lem:shiustrings},
there is an infinite set
$\cS\subseteq\NN$ with the property that
\begin{equation}
\label{eq:hungry1}
p_{r+1}\equiv p_{r+2}\equiv\cdots\equiv p_{r+m(r)}\equiv c\bmod d
\qquad(r\in\cS),
\end{equation}
where $m(r)$ satisfies \eqref{eq:m(r)bd}.
Taking into account \eqref{eq:greenearth_asymp},
we derive the following estimate for all $r\in\cS$:
\begin{align*}
\pi_\cP(p_{r+m(r)})-\pi_\cP(p_r)
&=\delta\bigl(\pi(p_{r+m(r)})-\pi(p_r)\bigr)
+O\bigl((\log\log r)^{\eps_r}\bigr)\\
&=\delta\,m(r)+O\bigl((\log\log r)^{\eps_r}\bigr)
\end{align*}
where $\eps_r\defeq\eps(p_{r+m(r)})$ for each $r$, and
the constant implied by the $O$-symbol depends only
on $\cP$. In view of \eqref{eq:m(r)bd} and the fact that
$\varlimsup\limits_{r\to\infty} \eps_r\le 0$, we have
\begin{equation}
\label{eq:hungry2}
\pi_\cP(p_{r+m(r)})>\pi_\cP(p_r)\qquad(r\in\cS,~r\ge r_0).
\end{equation}
For every sufficiently large $r\in\cS$,
by \eqref{eq:hungry2} it follows that
$p_{r+j}\in\cP$ for some $j$ in the range $1\le j\le m(r)$,
and by \eqref{eq:hungry1} we have $p_{r+j}\equiv c\bmod d$.
Since $\cS$ is infinite, the lemma follows.
\end{proof}

Using the Hardy-Littlewood circle method, Vinogradov~\cite{Vinogradov}
established his famous theorem that every sufficiently large odd integer
is the sum of three prime numbers. Effective versions of Vinogradov's theorem
have been given by several authors (see~\cite{Helfgott,Ramare,Tao} and references therein),
but for the purposes of the present paper we require only the following
extension of Vinogradov's theorem, which is due to Haselgrove~\cite[Theorem~A]{Haselgrove}.

\begin{lemma}
\label{lem:Hasel}
For any fixed $\theta\in(\tfrac{63}{64},1)$ there is a positive number $n_0(\theta)$
such that every odd integer $n\ge n_0(\theta)$ can be expressed as the sum
of three primes
$$
n=p_1+p_2+p_3
$$
with $|p_j-\tfrac13 n|<n^\theta$ for each $j=1,2,3$.
\end{lemma}

The following statement is a simple consequence of Haselgrove's result.

\begin{lemma}
\label{lem:vinny}
For every integer $s\ge 6$, there is an integer $N_0(s)$
with the property that every integer $N\ge N_0(s)$ can be
expressed as a sum of primes
$$
N=\widetilde p_1+\cdots+\widetilde p_s
$$
with $\widetilde p_j=2$ or $\widetilde p_j\ge\frac{1}{12}N$
for $j=1,\ldots,s$.
\end{lemma}

\begin{proof}
Set $\theta\defeq \frac{99}{100}$.  Since $\theta\in(\tfrac{63}{64},1)$,
Lemma~\ref{lem:Hasel} shows that there is a positive number $n_0=n_0(\theta)$
such that every odd integer $n\ge n_0$ can be expressed as the sum
of three primes, $n=p_1+p_2+p_3$, with $p_j\ge \tfrac14 n$
for each $j=1,2,3$.

Put $N_1(s)\defeq n_0+2s-6$, and let $N$ be an odd integer
exceeding $N_1(s)$. Since $n\defeq N-2s+6$ is an odd integer exceeding $n_0$,
we can write $n=p_1+p_2+p_3$ as above.  Consequently,
$$
N=p_1+p_2+p_3+
\mathop{\underbracket{\hskip3pt 2+\cdots+2\hskip3pt}}\limits_{\text{$s-3$ copies}}
$$
where
\begin{equation}
\label{eq:pjone}
p_j\ge\tfrac14 n=\tfrac14(N-2s+6)\qquad(j=1,2,3).
\end{equation}
Replacing $N_1(s)$ by a larger number, if necessary,
the bound $N>N_1(s)$ and \eqref{eq:pjone} together imply
that $p_j\ge\frac{1}{12}N$ for $j=1,2,3$.

Next, put $N_2(s)\defeq 3n_0+6s-36$, and let $N$ be an even integer
exceeding $N_2(s)$.  If $n_0$ is sufficiently large (which we can assume)
then $N-6s+36=n+n'$ for some odd integers $n$ and $n'$
that are both larger than $\max\{n_0,\tfrac13N\}$.  Therefore,
writing $n=p_1+p_2+p_3$ and $n'=p'_1+p'_2+p'_3$ as above, we have
$$
N=p_1+p_2+p_3+p'_1+p'_2+p'_3+
\mathop{\underbracket{\hskip3pt 2+\cdots+2\hskip3pt}}\limits_{\text{$s-6$ copies}}
$$
where
$$
p_j\ge\tfrac14 n\ge\tfrac{1}{12}N\mand
p'_j\ge\tfrac14 n'\ge\tfrac{1}{12}N
\qquad(j=1,2,3).
$$

Taking $N_0(s)\defeq\max\{N_1(s),N_2(s)\}$ we finish the proof.
\end{proof}

\begin{proof}[Proof of Theorem~\ref{thm:main_asymp}]
Fix a set $\cP\in\sB(\delta,\eps)$ with $2\in\cP$.
Since $\pi_\cP(x)$ satisfies \eqref{eq:greenearth_asymp} the first
condition of Lemma~\ref{lem:BanGulVau} is met, and it remains
only to verify the second condition of Lemma~\ref{lem:BanGulVau}.

 Fix an arbitrary integer $s\ge 6$,
and let $a,b\in\NN$ be given.  Replacing
$a$ with a sufficiently large number in the  progression
$a\bmod b$, we can assume that $a\ge 24b$.  We can further
assume that $a$ exceeds the number $N_0(s)$ described
in the statement of Lemma~\ref{lem:vinny}.  Therefore, $a$ can be expressed as a
sum of primes $a=\widetilde p_1+\cdots+\widetilde p_s$,
where for every $j=1,\ldots,s$ we have either
$\widetilde p_j=2$ or else
$$
\widetilde p_j\ge\tfrac{1}{12}a\ge 2b>b.
$$
In the latter case, it is clear that
$\gcd(\widetilde p_j,b)=1$, hence by
Proposition~\ref{prop:statue} there is a prime $p_j\in\cP$ such that
\begin{equation}
\label{eq:cong}
p_j\equiv\widetilde p_j\bmod b.
\end{equation}
Since $2\in\cP$, we can put $p_j\defeq 2$ whenever
$\widetilde p_j=2$, obtaining \eqref{eq:cong} in this case as well.
Summing the congruences \eqref{eq:cong} over $j=1,\ldots,s$
gives
$$
a=\widetilde p_1+\cdots+\widetilde p_s
\equiv p_1+\cdots+ p_s\bmod b.
$$
This shows that the second condition of Lemma~\ref{lem:BanGulVau} is met,
and the proof of Theorem~\ref{thm:main_asymp} is complete.
\end{proof}

\section*{Acknowledgements}

The author wishes to thank Tristan Freiberg, Andrew Granville, Victor Guo and Stephen Montgomery-Smith
for their helpful comments.


\begin{thebibliography}{9999}

\bibitem{BanFreTurB}
W.~D.~Banks, T.~Freiberg and C.~L.~Turnage-Butterbaugh,
``Consecutive primes in\break tuples,''
\emph{Acta Arith.} \textbf{167} (2015), no.~3, 261--266.

\bibitem{BanGulVau}
W.~D.~Banks, A.~M.~G\"ulo\u glu and R.~C.~Vaughan,
``Waring's problem for Beatty sequences and a local to global principle,''
\emph{J.\ Th\'eor.\ Nombres Bordeaux} \textbf{26} (2014), no.~1, 1--16. 

\bibitem{ChatSound}
S.~Chatterjee and K.~Soundararajan,
``Random multiplicative functions in short intervals,''
\emph{Int.\ Math.\ Res.\ Not.} (2012), no.~3, 479--492. 

\bibitem{GranSound}
A.~Granville and K.~Soundararajan,
``The distribution of values of $L(1,\chi_d)$,''
\emph{Geom.\ Funct.\ Anal.} \textbf{13} (2003), no.~5, 992--1028.
 
\bibitem{GreTao}
B.~Green and T.~Tao,
``Linear equations in primes,''
\emph{Ann.\ of Math.\ (2)} 171 (2010), no.~3, 1753--1850. 

\bibitem{HalbRoth}
H.~Halberstam and K.~F.~Roth,
\textit{Sequences}. Second edition.
Springer-Verlag, New York-Berlin, 1983.

\bibitem{Harper}
A.~J.~Harper,
``On the limit distributions of some sums of a random multiplicative function,''
\emph{J.\ Reine Angew.\ Math.} \textbf{678} (2013), 95--124. 

\bibitem{HarNikRad}
A.~J.~Harper, A.~Nikeghbali and
M.~Radziwi$\nshortmid$\hskip-2.5pt l$\nshortmid$\hskip-2.5pt l, 
``A note on Helson's conjecture on moments of random multiplicative functions,''
preprint, 2015.\quad{\tt arXiv:1505.01443}

\bibitem{Haselgrove}
C.~B.~Haselgrove,
``Some theorems in the analytic theory of numbers,''
\emph{J.\ London Math.\ Soc.} \textbf{26} (1951), 273--277. 

\bibitem{Helfgott}
H.~A.~Helfgott,
``Chebotarev sets,''
2011.  Preprint available on arXiv: {\tt arXiv:1312.7748}

\bibitem{Khintchine}
A.~Khintchine,
``\"Uber einen Satz der Wahrscheinlichkeitsrechnung,''
\emph{Fund.\ Math.} \textbf{6} (1924), 9--20.
 
\bibitem{KisRub}
H.~Kisilevsky and M.~O.~Rubinstein,
``Chebotarev sets,''
2011.  Preprint available on arXiv: {\tt arXiv:1112.4945}

\bibitem{LauTenWu}
Y.-K.~Lau, G.~Tenenbaum and J.~Wu,
``On mean values of random multiplicative functions,''
\emph{Proc.\ Amer.\ Math.\ Soc.} \textbf{141} (2013), no.~2,
409--420. 

\bibitem{LiPan}
H.~Li and H.~Pan,
``A density version of Vinogradov's three primes theorem,''
\emph{Forum Math.} 22 (2010), no.~4, 699--714. 

\bibitem{Mat}
K.~Matom\"aki,
``Sums of positive density subsets of the primes,''
\emph{Acta Arith.} 159 (2013), no.~3, 201--225. 

\bibitem{Maynard}
J.~Maynard,
``Small gaps between primes,''
\emph{Ann.\ of Math.\ \textup{(}2\textup{)}} 
\textbf{181} (2015), no.~1, 383--413.

\bibitem{Ramare}
O.~Ramar\'e,
``On $\check{\rm S}$nirel'man's constant,''
\emph{Ann.\ Scuola Norm.\ Sup.\ Pisa Cl.\ Sci.\ (4)} \textbf{22} (1995), no.~4, 645--706. 

\bibitem{RamRus}
O.~Ramar\'e and I.~Z.~Ruzsa,
``Additive properties of dense subsets of sifted sequences,''
\emph{J.\ Th\'eor.\ Nombres Bordeaux} 13 (2001), no.~2, 559--581. 

\bibitem{Riemann}
B.~Riemann, ``Ueber die Anzahl der Primzahlen unter einer gegebenen Gr\"osse,''
\emph{Monatsberichte der Berliner Akademie}, 1859.

\bibitem{Sar}
A.~S\'ark\"ozy,
``On finite addition theorems.
Structure theory of set addition.''
\emph{Ast\'erisque} 258 (1999), nos.~xi-xii, 109--127. 

\bibitem{Shao}
X.~Shao,
``A density version of the Vinogradov three primes theorem,''
\emph{Duke Math.\ J.} 163 (2014), no.~3, 489--512. 

\bibitem{Shiu}
D.~K.~L.~Shiu,
``Strings of congruent primes,''
\emph{J.\ London Math.\ Soc.\ (2)} \textbf{61} (2000), no.~2, 359--373. 

\bibitem{Tao}
T.~Tao,
``Every odd number greater than $1$ is the sum of at most five primes,''
\emph{Math.\ Comp.} \textbf{83} (2014), no.~286, 997--1038. 

\bibitem{Vinogradov}
I.~M.~Vinogradov,
``Representation of an odd number as a sum of three primes,''
\emph{Comptes Rendus (Doklady) de l'Academy des Sciences de l'USSR} \textbf{15} (1937), 191--294.

\bibitem{Wintner}
A.~Wintner,
``Random factorizations and Riemann's hypothesis,''
\emph{Duke Math.\ J.} \textbf{11} (1944), 267--275. 
\end{thebibliography}
\end{document}